\documentclass[10pt]{article}
\usepackage{amsmath,amssymb,amsthm}
\usepackage{epsfig}
\usepackage{color}

\title{Nordhaus-Gaddum type inequalities for multiple domination and packing parameters in graphs}

\author {
Doost Ali Mojdeh and Babak Samadi\\
Department of Mathematics\\
University of Mazandaran, Babolsar, Iran\\
{\tt damojdeh@umz.ac.ir}\\
{\tt samadibabak62@gmail.com}\vspace{3mm}\\
Lutz Volkmann\\
Lehrstuhl II f\"{u}r Mathematik\\
RWTH Aachen University, 52056 Aachen, Germany\\
{\tt volkm@math2.rwth-aachen.de}\vspace{3mm}\\
}

\date{}

 \newtheorem{theorem}{Theorem}[section]

\theoremstyle{definition}

\begin{document}

\maketitle
\begin{abstract}
\noindent We study the Nordhaus-Gaddum type results for $(k-1,k,j)$ and $k$-domination numbers of a graph $G$ and 
investigate these bounds for the $k$-limited packing and $k$-total limited packing numbers in graphs. As the special 
case $(k-1,k,j)=(1,2,0)$ we give an upper bound on $dd(G)+dd(\overline{G})$ stronger than that presented by 
Harary and Haynes (1996). Moreover, we establish upper bounds on the sum and product of packing and open packing 
numbers and characterize all graphs attaining these bounds.
\end{abstract}
\vspace{1mm}
\section{Introduction}

Throughout this paper, let $G$ be a finite graph with vertex set $V=V(G)$, edge set $E=E(G)$, minimum degree 
$\delta=\delta(G)$ and maximum degree $\Delta=\Delta(G)$. We use \cite{w} as a reference for terminology and notation 
which are not defined here. For any vertex $v\in V$, $N(v)=\{u\in G\mid uv\in E(G)\}$ denotes the {\em open neighborhood} 
of $v$ in $G$, and $N[v]=N(v)\cup \{v\}$ denotes its {\em closed neighborhood}. We denote the {\em disjoint union} of two 
graphs $G$ and $H$ by $G+H$. The complement $\overline{G}$ of a graph $G$ has vertex set $V(G)$ and 
$uv‎\in E(\overline{G})‎$ if and only if $uv‎\notin E(G)$.\\
Nordhaus and Gaddum \cite{ng}, gave lower and upper bounds on the sum and product of the chromatic number of a graph and its complement, in terms of the order of the graph. Since then, bounds on $‎\psi(G)+‎\psi(\overline{G})‎‎$ or $\psi(G)\psi(\overline{G})‎‎$ are called Nordhaus-Gaddum inequalities, where $‎\psi‎$ is a graph parameter. For more information about this subject the reader can consult \cite{ah}.\\
A set  $S\subseteq V$ is a {\em dominating set} ({\em total dominating set}) in $G$ if each vertex in $V \setminus S$ (in $V$) is adjacent to at least one vertex in $S$. The {\em domination number} $\gamma(G)$ ({\em total domination number} $\gamma_{t}(G)$) is the minimum cardinality of a dominating set (total dominating set) in $G$. A subset $B \subseteq V$ is a {\em packing} ({\em open packing}) in $G$ if for every distinct vertices $u,v‎\in B‎$, $N[u]‎‎\cap N[v]‎=‎\emptyset‎$ ($N(u)‎‎\cap N(v)=‎\emptyset‎‎‎$). The {\em packing number} ({\em open packing number}) $\rho(G)$ ($‎\rho‎‎‎‎_{o}‎(G)‎‎$) is the maximum cardinality of a packing (an open packing) in $G$.  Clearly, $B \subseteq V$ is a packing (an open packing) in $G$ if and only if $|N[v]‎\cap B‎|‎\leq1‎$ ($|N(v)‎\cap B‎|‎\leq1‎$), for all $v‎\in V‎$. Here, we prefer to work with these definitions rather than the previous ones.\\
Let $k,k'$ and $k''$ be nonnegative integers. A set  $S\subseteq V$ is a {\em $(k,k',k'')$-dominating set} in $G$ if every vertex in $S$ has at least $k$ neighbors in $S$ and every vertex in $V‎\setminus S‎$ has at least $k'$ neighbors in $S$ and at least $k''$ neighbors in $V‎\setminus S‎$. The {\em $(k,k',k'')$-domination number} $\gamma‎_{(k,k',k'')}‎(G)$ is the minimum cardinality of a $(k,k',k'')$-dominating set. We note that every graph with the minimum degree at least $k$ has a $(k,k',k'')$-dominating set, since $S=V(G)$ is such a set. This concept was introduced in \cite{skg}, as a generalization of many domination parameters.\vspace{1mm}\\
Note that
\begin{itemize}
\item $\gamma‎_{(0,1,1)}‎(G)=\gamma‎_{r}(G)$: {\em Restrained domination number} (\cite{tp});
\item $\gamma‎_{(1,2,1)}‎(G)=‎\gamma‎_{2r}‎‎(G)$: {\em Restrained double domination number} (\cite{kn});
\item $\gamma‎_{(k-1,k,0)}‎(G)=\gamma_{‎\times ‎k}‎(G): k$-{\em Tuple domination number} (\cite{gghr,hh1,sdk});
\item $\gamma‎_{(0,k,0)}‎(G)=‎\gamma‎_{k}(G)‎‎: k$-{\em Domination number}. (\cite{cfhv,fj1,fj2,p,v1})
\end{itemize}
A subset $S‎\subseteq V(G)‎$ is $k$-{\em independent} if the maximum degree of the subgraph induced by the vertices of $S$ is less or equal to $k-1$. A subset $B \subseteq V$ is a {\em $k$-limited packing} ({\em $k$-total limited packing}) in $G$ if $|N[u]‎\cap B‎|‎\leq k‎‎$ ($|N(u)‎\cap B‎|‎\leq k‎‎$), for every vertex $u\in V$. The {\em $k$-limited packing number} ({\em $k$-total limited packing number}) $L‎_{k}‎(G)$ ($L‎_{k,t}‎(G)$) is the maximum cardinality of a $k$-limited packing ($k$-total limited packing) in $G$. These concepts were introduced and investigated in \cite{gghr} and \cite{hms}, respectively.\\
In this paper, we continue the study of these bounds for the above domination parameters. We give an upper bound on the sum of $(k-1,k,j)$-domination number of a graph and its complement which gives improvements of a conjecture by Harary and Haynes \cite{hh2}. Moreover, we continue presenting Nordhaus-Gaddum bounds for $k$-limited packing and $k$-total limited packing numbers with emphasis on the case $k=1$. Furthermore, we characterize all graphs attaining these bounds. This subject was initiated by exhibiting the upper bound $L‎_{2}(G)+L‎_{2}(\overline{G})‎‎\leq n+2‎$ in \cite{s}.


\section{$(k-1,k,j)$ and $k$-domination}

Harary and Haynes \cite{hh2} established the following Nordhaus-Gaddum inequality for double domination number when $‎\gamma(G),‎\gamma(\overline{G})‎\geq5$:
\begin{equation}\label{EQ1}
dd(G)+dd(\overline{G})‎\leq ‎\delta(G)+‎\delta(\overline{G})‎‎‎
\end{equation}
Also, they conjectured that for any graph $G$ with $‎\gamma(G),‎\gamma(\overline{G})‎\geq4$
\begin{equation}\label{EQ2}
dd(G)+dd(\overline{G})‎\leq ‎\delta(G)+‎\delta(\overline{G})‎‎‎
\end{equation}
In \cite{sdk}, the conjecture was generalized and proved as
\begin{equation}\label{EQ4}
\gamma‎_{‎\times k}(G)‎‎‎+\gamma‎_{‎\times k}(\overline{G})‎‎‎‎\leq \delta(G)+‎\delta(\overline{G})
\end{equation}
when $‎\gamma(G),‎\gamma(\overline{G})‎\geq k+2$.\\
The following theorem improves the upper bounds (\ref{EQ2}) and (\ref{EQ4}). Moreover, as the special case $k=2$ it leads to the following upper bound which is stronger than (\ref{EQ1}):
\begin{center}
$dd‎‎(G)+dd(\overline{G})‎\leq \delta(G)+‎\delta(\overline{G})-(\gamma(G)+\gamma(\overline{G}))‎+8‎\leq‎‎\delta(G)+‎\delta(\overline{G})-2‎‎‎$,
\end{center}
when $‎\gamma(G),‎\gamma(\overline{G})‎\geq5$.
\begin{theorem}
Let $j\ge 0$ and $k\ge 1$ be two integers. If $G$ is a graph with
$‎\gamma(G),‎\gamma(\overline{G})‎\geq k+j+2$, then
$$\gamma‎_{‎(k-1,k,j)}(G)‎‎‎+\gamma‎_{‎(k-1,k,j)}(\overline{G})‎‎‎‎\leq \delta(G)+‎\delta(\overline{G})-(\gamma(G)+\gamma(\overline{G}))‎+2k+4‎‎‎.$$
\end{theorem}

\begin{proof}
Let $u$ be a vertex of the minimum degree $‎\delta(G)‎$ and $v_{0}$ be another vertex of $G$.
Since $‎\gamma(\overline{G})‎\geq k+j+2\ge 3$, the set $\{u,v_{0}\}$ does not dominate $V(G)$ in
$\overline G$. Thus $W_{0}=V(G)-N_{\overline G}[\{u,v_{0}\}]\neq\emptyset$. Let $X_{0}\subseteq W_{0}$ be an independent set
of maximum size in $\overline G$. Then it is easy to see that $X_{0}‎\cup‎ \{u,v_{0}\}$ is a dominating set in
$\overline G$. Therefore $k+j‎\leq‎‎\gamma(\overline{G})‎-2‎\leq|X_{0}|‎$ and $u$ and $v_{0}$
have at least $‎\gamma(\overline{G})‎-2$ mutually adjacent common neighbors in $N_G(u)$.\\
Let $X_{0}'$ be a subset of $X_{0}$ with $|X_{0}'|=‎\gamma(\overline{G})‎-2‎‎$ and
\begin{center}
$D=N_G(u)‎\setminus(X_{0}'‎\setminus\{x‎_{1}‎,x_2,...,x‎_{k}‎\}‎)‎$,
\end{center}
in which $x‎_{1},x_2,...,x‎_{k}$ are arbitrary vertices of $X_{0}'$. Obviously, $u$ is adjacent to 
$x‎_{1},x_2,...,x‎_{k}$ in $D$ and $|N_{G}(x_{i})\cap D|\geq k-1$, for all $1\leq i\leq k$. Now let 
$v\neq u,x‎_{1}‎,x_{2},...,x‎_{k}$ be an arbitrary vertex of $G$. Similar to the above argument, 
there exists a set $X\subseteq N_{G}(v)$, with $|X|\geq ‎\gamma(\overline{G})-2$ and $X\subseteq N_{G}(u)$. Thus
\begin{equation*}
\begin{array}{lcl}
|N_{G}(v)\cap D|&=&|N_{G}(v)\cap N_{G}(u)|-|N_{G}(v)\cap(X_{0}'‎\setminus\{x‎_{1}‎,x_2,...,x‎_{k}‎\})|\\
&\geq& \gamma(\overline{G})‎-2-|X_{0}'‎\setminus\{x‎_{1}‎,x_2,...,x‎_{k}‎\}|\\
&=& k.
\end{array}
\end{equation*}
So, all vertices in $V(G)\setminus\{x‎_{1}‎,x_2,...,x‎_{k}‎\}$ must have at least $k$ neighbors in $D$.\\
On the other hand, every vertex in $V(G)‎\setminus D‎$ has at most $|D|‎\leq ‎\delta-j$ neighbors
in $D$ and so at least $j$ neighbors in $V‎\setminus D‎$.\\
The above argument shows that $D$ is a $(k-1,k,j)$-dominating set in $G$.
Thus,
\begin{equation}\label{EQ6}
‎\gamma‎_{‎(‎k-1,k,j)}(G)‎‎‎\leq|D|‎‎=|N(u)‎\setminus(X_{0}'‎\setminus\{x‎_{1}‎,...,x‎_{k}\}‎)‎|‎‎=‎\delta(G)-  \gamma(\overline{G})‎‎‎+k+2.
\end{equation}
By the symmetry between $G$ and $\bar G$, we have
\begin{equation}\label{EQ7}
\gamma‎_{‎‎(k-1,k,j)}(\overline{G})‎‎‎‎\leq \delta(\overline{G})-‎‎‎‎\gamma(G)+k+2‎.
\end{equation}
Now we deduce from the inequalities (\ref{EQ6}) and (\ref{EQ7}) that
$$\gamma‎_{‎‎(k-1,k,j)}(G)‎‎‎+\gamma‎_{‎(‎k-1,k,j)}(\overline{G})‎‎‎‎\leq ‎\delta(G)+‎\delta(\overline{G})‎‎‎-(\gamma(G)+\gamma(\overline{G}))‎+2k+4,$$
as desired.
\end{proof}

Nordhaus-Gaddum bounds for $(2‎\leq)k‎$-domination number are not known as many as those for the usual domination number (the case $k=1$). In fact, Volkmann \cite{v1} showed that
\begin{center}
$\gamma‎_{2}(G)‎‎‎+\gamma‎_{2}(\overline{G})‎‎‎‎\leq‎ n+2$, for a graph $G$ of order $n$.
\end{center}
Moreover, Prince \cite{p} proved the following upper bound
\begin{center}
$\gamma‎_{k}(G)‎‎‎+\gamma‎_{k}(\overline{G})‎‎‎‎\leq‎ n+2k-1$, for a graph $G$ of order $n$.
\end{center}
We improve these two upper bounds for the case when $‎\gamma(G),\gamma(\overline{G})‎‎\geq k+2‎$ as follows.
\begin{theorem}
Let $k\ge 1$ be an integer. If $G$ is a graph with $‎\gamma(G),\gamma(\overline{G})‎‎\geq k+2‎$, then
$$\gamma‎_{k}(G)‎‎‎+\gamma‎_{‎k}(\overline{G})‎‎‎‎\leq
‎\kappa‎(G)+‎\kappa‎(\overline{G})-(\gamma(G)+\gamma(\overline{G}))‎+2k+4‎‎‎.$$
\end{theorem}
\begin{proof}
Let $A‎\subseteq V(G)‎$ be a vertex cut of $G$ with $|A|=\kappa(G)$, and let $u$ and $v$ be two vertices
from different components of $G-A$. Similar to the proof of Theorem 2.1, $u$ and $v$ have at least
$‎\gamma(\overline{G})‎-2‎\geq k‎$ mutually adjacent common neighbors and these neighbors must
be in $A$. Let $X‎\subseteq A‎$ be a set of size $‎\gamma(\overline{G})‎-2‎‎‎$
such that $G[X]$ is complete. We define
\begin{center}
$S=A‎\setminus(X‎\setminus\{x‎_{1}‎,x_2,...,x‎_{k}‎\}‎)‎$,
\end{center}
where $x‎_{1}‎,x_2,...,x‎_{k}$ are arbitrary vertices in $X$. Clearly, every vertex in
$X‎\setminus\{x‎_{1}‎,x_2,...,x‎_{k}‎\}$ is adjacent to
$x‎_{1}‎,x_2,...,x‎_{k}\in S$.  Moreover, all vertices in $V‎\setminus A$ must have at least
$‎\gamma(\overline{G})‎-2‎‎$ neighbors in $A$, and at least $k$ of them must be in $S$. Therefore
$S$ is a $k$-dominating set in $G$. It follows that
\begin{equation}\label{EQ8}
\gamma‎_{k}(G)‎\leq|S|‎=‎\kappa(G)‎-‎\gamma(\overline{G})‎+k+2.
\end{equation}
By the symmetry, we have
\begin{equation}\label{EQ9}
\gamma‎_{k}(\overline{G})\leq‎\kappa(\overline{G})-‎\gamma(G)‎‎+k+2.
\end{equation}
Adding the inequalities (\ref{EQ8}) and (\ref{EQ9}), we obtain the desired upper bound.
\end{proof}


\section{Packing and open packing}

Since $‎\rho‎_{o}(G)+\rho‎_{o}(\overline{G})‎‎‎‎=n+1$ and $‎\rho‎_{o}(G)\rho‎_{o}(\overline{G})‎‎‎‎=n$ (except when $n=2$, for which $‎\rho‎_{o}(G)+\rho‎_{o}(\overline{G})=‎\rho‎_{o}(G)\rho‎_{o}(\overline{G})=n+2$), for every graph $G$ of order $n$ with $‎\Delta(G)‎$ or $‎\Delta(\overline{G})‎=0$, we may always assume that $‎\Delta(G)‎,‎\Delta(\overline{G})‎\geq1‎‎$. First, we define $‎\Gamma‎$ to be the family of all graphs $G$ satisfying:\vspace{1mm}\\
$(i)$\ There exists a vertex $v$ in the open neighborhood $N(u)$, of a vertex $u$ of the maximum degree $‎\Delta(G)‎$, such that $N[v]\subseteq N[u]‎$;\vspace{1mm}\\
$(ii)$\ The subset $V‎(G)\setminus N[u]‎$ is an independent set;\vspace{1mm}\\
$(iii)$\ Every vertex in $N[u]\setminus N[v]‎$ has at most one neighbor in $V(G)‎\setminus N[u]‎$.\vspace{1mm}\\
We are now in a position to present the following theorem.
\begin{theorem}
Let $G$ be a graph of order $n$ with $‎\Delta(G),‎\Delta(\overline{G})‎‎\geq1‎$. Then
\begin{equation*}\label{EQ10}
‎\rho‎(G‎)+‎\rho(\overline{G})‎‎‎\leq‎ \left\{
\begin{array}{ccc}
n-‎\Delta(G)+1‎& \text{if } & ‎\gamma(\overline{G})‎=1‎‎‎‎\\
n-‎\Delta(G)+2‎ & \text{if } & \gamma(\overline{G})‎=2\\
‎\delta(G)+2‎‎ & \text{if } & \gamma(\overline{G})‎‎\geq‎3\\
\end{array} \right.
\end{equation*}
and
\begin{equation*}\label{EQ11}
‎\rho‎(G‎)\rho(\overline{G})‎‎‎\leq‎ \left\{
\begin{array}{ccc}
n-‎\Delta(G)‎& \text{if } & ‎\gamma(\overline{G})‎=1‎‎‎‎\\
2n-2‎\Delta(G)‎ & \text{if } & \gamma(\overline{G})‎=2\\
\delta(G)+1‎ & \text{if } & \gamma(\overline{G})‎‎\geq‎3\\
\end{array} \right.
\end{equation*}
The equality in the case $‎\gamma(\overline{G})‎=1$ and $\gamma(\overline{G})‎‎\geq3$ holds if and only if $G\in ‎\Gamma‎$ and $\overline{G}\in ‎\Gamma‎$, respectively. Furthermore, the equality in the cases $\gamma(\overline{G})‎=2$ holds if and only if $G\in ‎\Gamma‎$ and $diam(\overline{G})‎\geq‎3‎$.
\end{theorem}
\begin{proof}
Let $B$ be a maximum packing in $G$ and $u$ be a vertex in $V(G)$ of maximum degree $‎\Delta(G)‎$. Since $B$
is a packing in $G$, at most one vertex in $N[u]$ is in $B$. So,
\begin{equation}\label{EQ12}
\rho(G)=|B|‎\leq n-‎\Delta(G)‎‎‎.
\end{equation}
In what follows, we prove that $‎\rho(G)=n-‎‎\Delta(G)‎$ if and only if $G\in ‎\Gamma‎$. Assume first that we have the equality. Then $|N[u]‎\cap B‎|=1$, otherwise $‎\rho(G‎‎)‎\leq n-‎\Delta‎‎‎-1$ and this is a contradiction. Moreover, $V(G)‎\setminus N[u]‎\subseteq B‎‎$. Let $v\in B‎\cap N[u]‎$. By the definition of the packing set we have $N[v]‎\cap(V(G)‎\setminus N[u])‎=‎\phi‎$. This implies $(i)$. By definition of $B$ and since $V(G)‎\setminus N[u]‎\subseteq B‎‎$, we derive at $(ii)$ and $(iii)$.\\
Now let $G\in ‎\Gamma‎$. Then $B'=\{v\}‎\cup(V(G)‎\setminus N[v])‎$ is a packing in $G$, by $(i)$, $(ii)$ and $(iii)$. Therefore, $‎\rho(G)‎\geq|B'|‎\geq n-‎\Delta(G)‎‎‎‎$ and hence the equality holds.\\
We now distinguish two cases depending on $\gamma(\overline G)‎$.\\
{\bf Case 1.} Let $\gamma(\overline{G})=1‎$ or $2$. Since, $‎\rho(\overline{G})‎\leq \gamma(\overline{G})$ (see \cite{gghr}) by (\ref{EQ12}) we have $‎\rho‎(G‎)+‎\rho(\overline{G})‎‎‎\leq n-‎\Delta(G)+1‎$ or $n-‎\Delta(G)+2‎$ and $‎\rho‎(G‎)‎\rho(\overline{G})‎‎‎\leq n-‎\Delta(G)‎$ or $2n-‎2\Delta‎(G)$, respectively.\vspace{0.1mm}\\
{\bf Case 2.} Let $\gamma(\overline{G})‎\geq3‎‎$. Let $u$ and $v$ be two distinct vertices in $B$.
Since $‎\gamma(\overline{G})‎\geq3$, the set $\{u,v\}$ does not dominate $V(G)$ in $\overline G$.
So there exists a vertex $w$ in $V(G)‎$ with $N‎_{\overline{G}}‎(w)‎\cap \{u,v\}=‎\emptyset$.
Therefore $|N_G[w]‎\cap B‎|‎\geq|N_G[w]‎\cap\{u,v\}‎|=2‎$, a contradiction.
Hence $‎\rho(G)=1‎$. Considering the symmetry between $G$ and $\overline{G}$ and (\ref{EQ12}) we have
$\rho‎(G‎)+‎\rho(\overline{G})‎‎‎\leq n-‎\Delta(\overline{G})‎+1
=‎\delta(G)+2‎$ and
$\rho‎(G‎)‎\rho(\overline{G})‎‎‎
\leq n-‎\Delta(\overline{G})‎=‎\delta(G)+1‎$.\\
The second part of the theorem follows by considering all graphs attaining the upper bound (\ref{EQ12}) and the fact that
$‎\rho(G)=1‎$ if and only if diam$(G)‎\leq2‎$, for each graph $G$. This completes the proof.
\end{proof}
We now turn our attention to the analogous problem for the parameter $‎\rho‎_{o}(G)$. We define $‎\Pi‎$ to be the family of all graphs $G$ for which $‎\Delta(G)=|V(G)|-1‎$ and $‎\delta(G)=1‎$. We make use of this class of graphs when we characterize the extremal graphs corresponding to the upper bounds in the next theorem.
\begin{theorem}
Let $G$ be a graph of order $n$. If $‎\gamma(\overline{G})‎‎\geq3‎$, then
$$‎\rho‎_{o}(G)+‎\rho‎_{o}(\overline{G})‎\leq ‎\delta(G)+3‎‎‎‎‎ \ \ \ \&\ \ \ \rho‎_{o}(G)‎\rho‎_{o}(\overline{G})‎\leq ‎\delta(G)+2.$$
Furthermore, the upper bounds hold with equality if and only if $\overline{G}$ is isomorphic to $H+rK‎_{2}‎+sK‎_{1}‎$ for some non-negative integers $r$ and $s$, where $H\in ‎\Pi‎$.
\end{theorem}

\begin{proof}
Let $B$ be an open packing in $G$ of maximum size and $u$ be a vertex of maximum degree $‎\Delta(G)‎$. Then at most two vertices in $N[u]$ belong two $B$ and one of them must be $u$, necessarily. Thus,
\begin{equation}\label{EQ13}
\rho‎_{o}‎(G)=|B|‎\leq n-‎\Delta(G)‎‎‎+1.
\end{equation}
We now show that the equality in (\ref{EQ13}) holds if and only if $G=H+rK‎_{2}‎+sK‎_{1}‎$ for some non-negative integers $r$ and $s$, where $H\in ‎\Pi‎$. Let the equality holds for the graph $G$. If $u$ is a vertex of the maximum degree $‎\Delta(G)‎$, then there exists two vertices in $N[u]‎\cap B‎$, otherwise $‎\rho‎_{o}(G)‎\leq n-‎\Delta(G)$ and this is a contradiction. On the other hand, by the definition of the open packing one of these two vertices is $u$ and the other must be a pendant vertex adjacent to $u$, necessarily. Moreover, $V(G)‎\setminus N[u]‎‎\subseteq B‎$. This shows that $V(G)‎\setminus N[u]$ is $2$-independent and therefore it is isomorphic to disjoint unions of $K‎_{2}‎$ and $K‎_{1}‎$. Conversely, let $G$ be such a graph. Then every maximum packing in $G$ contains all vertices of the subgraph $rK‎_{2}‎+sK‎_{1}$ of $G$, the vertex of the maximum size and a pendant vertex of $H$. So, $\rho‎_{o}‎(G)=|B|‎=n-‎\Delta(G)‎‎‎+1$.\\
Since $‎\gamma(\overline{G}‎)‎\geq3‎$, a reason similar to one given in the proof of Theorem 3.1 shows that $‎\rho‎_{o}(G)=1‎‎$. Applying the inequality (\ref{EQ13}) to $\overline{G}$, we have $‎\rho‎_{o}(G)+‎\rho‎_{o}(\overline{G})‎\leq ‎n-\Delta(\overline{G})+2=‎\delta(G)‎‎+3$. Moreover, the equality holds if and only if $\rho‎_{o}‎(\overline{G})=n-‎\Delta(\overline{G})‎‎‎+1$. This completes the proof.
\end{proof}
We note that the condition $‎\gamma(\overline{G})‎‎\geq3‎$ in Theorem 3.2 implies that at least 
one of the integers $r$ and $s$ is positive.\\
In the next two theorems we establish upper bounds on the sum and product of the open packing number of a graph and 
its complement with no additional conditions.
\begin{theorem}
If $G$ is a graph of order $n‎\geq3‎$ with maximum degree $‎\Delta‎$ and minimum degree $‎\delta‎$, then
$$‎\rho‎_{o}(G)+‎\rho‎_{o}(\overline{G})‎\leq n-‎\Delta+‎\delta‎‎+3‎‎‎‎‎ \ \ \ \&\ \ \ \rho‎_{o}(G)‎\rho‎_{o}(\overline{G})‎\leq(n-‎\Delta+1‎)(‎\delta+2‎).$$
The bounds hold with equality if and only if $\{G,\overline{G}\}=\{H,H'+K‎_{1}‎\}$, where $H,H'\in ‎\Pi‎$.
\end{theorem}
\begin{proof}
Using (\ref{EQ13}) and the symmetry between $G$ and $\overline{G}$ we have
\begin{equation}\label{EQ14}
\begin{array}{lcl}
‎\rho‎_{o}(G)+‎\rho‎_{o}(\overline{G})‎\leq n-‎\Delta(G)+1+n-‎\Delta(\overline{G})‎+1=n-‎\Delta+‎\delta+3 \\
\ \ \ \ \ \ \ \ \ \ \ \ \ \ \ \ \ \ \ \ \ \ \ \ \ \ \ \ \ \ \ \ \ \ \ \ \ \ \ \ \ \ \mbox{and}\\
\rho‎_{o}(G)‎\rho‎_{o}(\overline{G})‎\leq(n-‎\Delta(G)+1‎)(n-‎\Delta(\overline{G})+1)=(n-‎\Delta+1‎)(‎\delta+2‎).
\end{array}
\end{equation}
On the other hand, by the proof of Theorem 3.2 the bounds given in (\ref{EQ14}) hold with equality if and only if
$G=H+rK‎_{2}‎+sK‎_{1}$‎ and $\overline G=H'+r'K‎_{2}‎+s'K‎_{1}$ for some
non-negative integers $r,s,r',s'$, where $H,H'\in ‎\Pi‎$. We assume that the upper bounds (\ref{EQ14}) hold with
equality. Assume first that $r>0$ and consider a copy of $K‎_{2}‎$ on two vertices $u$ and $v$, as a
component of $G-H‎$. Then $uv\notin E(\overline{G})$ and $u$ and $v$ are adjacent to all other 
$n-2‎\geq2‎$ vertices of $\overline{G}$. This shows that $\overline G$ is connected and 
$\delta(\overline G)\geq2$, a contradiction.
Therefore $r=0$. Moreover, by the symmetry we have $r'=0$. Now let $u$ be an isolated vertex as a component of $G-H$. Then, $|N‎_{\overline{G}}(u)|=n-1‎$ and hence $s'=0$ and $\overline{G}$ has at least one vertex of degree one. But if $s‎\geq2‎$, then there is no vertex of $\overline{G}$ of degree one and this is a contradiction. So, $s‎\leq1‎$. Also, $s'‎\leq1‎$ by the symmetry. On the other hand the cases $s=s'=1$ and $s=s'=0$ are impossible, by the constructions of $G$ and $\overline{G}$. Thus, $(s,s')\in\{(1,0),(0,1)\}$. This implies that $\{G,\overline{G}\}=\{H,H'+K‎_{1}‎\}$, where $H,H'\in ‎\Pi‎$.\\
Now let $\{G,\overline{G}\}=\{H,H'+K‎_{1}‎\}$, where $H,H'\in ‎\Pi‎$. Then the bounds given in (\ref{EQ14}) hold with equality, by the proof of Theorem 3.2.
\end{proof}

The first Nordhaus-Gaddum type inequality for the sum of the total domination numbers of a graph and its complement was given in \cite{cdh}. Henning et al. \cite{hjs} extend this result to include the product.
\begin{theorem}(\cite{cdh,hjs})
If $G$ is a graph of order $n$ such that neither $G$ nor $\overline{G}$ contains isolated vertices, then $‎\gamma‎_{t}(G)+\gamma‎_{t}(\overline{G})‎\leq n+2‎‎‎$ and $‎\gamma‎_{t}(G)\gamma‎_{t}(\overline{G})‎\leq2n‎‎‎$. Furthermore, the equality holds if and only if $G$ or $\overline{G}$ consists of disjoint copies of $K‎_{2}‎$.
\end{theorem}
We now give a Nordhaus-Gaddum bound for the sum and product of the open packing numbers of a graph and
its complement just in terms of its order.
\begin{theorem}
Let $G$ be a graph of order $n$. Then
$‎\rho‎_{o}(G)+‎‎‎\rho‎_{o}(\overline{G})‎\leq n+2‎$ and
$‎\rho‎_{o}(G)\rho‎_{o}(\overline{G})‎\leq2n‎$. Furthermore, these bounds hold with equality if
and only if $\{G,\overline G\}=\{2K_{2},C_{4}\}$ or $\{G,\overline{G}\}=\{K_{2},2K_{1}\}$ or
$\{G,\overline{G}\}=\{P‎_{3},\overline{P‎_{3}‎}\}$.
\end{theorem}
\begin{proof}
We consider two cases.\\
{\bf Case 1.} Let $G$ and $\overline{G}$ have no isolated vertices. Since
$\rho‎_{o}(H)‎\leq \gamma‎_{t}(H)‎$ for every graph $H$ with no isolated vertices (see \cite{r}),
the upper bounds follow by Theorem 3.4. Obviously, the equality holds for 
$\{G,\overline{G}\}=\{2K‎_{2},C_{4}‎‎\}$. Now let the upper bounds hold with the equality for the graph 
$G$. Since $\rho‎_{o}(G)‎\leq \gamma‎_{t}(G)‎$, we deduce from Theorem 3.4 that
$‎\gamma‎_{t}(G)+\gamma‎_{t}(\overline{G})‎=n+2‎‎‎$. Without loss of generality, 
we may assume that $G=‎\frac{n}{2}K‎_{2}‎‎$ ($n$ is necessarily even), by Theorem 3.4. This implies 
that $‎\rho‎_{o}(G)=n$ and $‎‎‎\rho‎_{o}(\overline{G})‎=2$. If $n‎\geq6‎$, 
then $‎\gamma(G)‎\geq3‎‎$ and similar to the proof of Theorem 3.1, we have 
$‎\rho‎_{o}(\overline{G})=1‎‎$, a contradiction. So, $n=2$ or $n=4$. Since neither $G$ nor 
$\overline{G}$ contains isolated vertices, we have $n=4$. So $G=2K_{2}$ and $\overline G=C_{4}$.\\
{\bf Case 2.} We now consider the case in which $G$ (or $\overline{G}$) has an isolated vertex $v$. Since
$N‎_{\overline{G}}[v]=V(\overline{G})‎$, we have $\rho‎_{o}(\overline{G})‎‎\leq2‎‎$.
This implies the upper bounds. Clearly, the equality holds for $\{G,\overline G\}=\{K_2,2K_1\}$ or
$\{G,\overline{G}\}=\{P‎_{3},\overline{P‎_{3}‎}\}$.  Now let the upper bounds hold with equality.
If $G$ has at least two isolated vertices, then $n=2$ or $‎\gamma(G)‎\geq3‎‎$. If
$n=2$, then $\{G,\overline{G}\}=\{K_{2},2K_{1}\}$. If $n\ge 3$, then $\gamma(G)\ge 3$ and therefore
$‎\rho‎_{o}(\overline{G})=1‎‎$, implying
$‎\rho‎_{o}(G)+‎‎‎\rho‎_{o}(\overline{G})‎‎‎\leq‎‎ n+1‎$ and
$‎\rho‎_{o}(G)‎‎‎\rho‎_{o}(\overline{G})‎‎‎\leq n‎$. These are
contradictions. So we may assume that $G$ has just one isolated vertex.
Taking into account the fact that the upper bounds hold with equality then
$\rho‎_{o}(\overline{G})‎‎=2‎‎$ and $‎\rho‎_{o}(G)=n$. Thus
$‎\Delta(G)‎\leq1‎‎$. On the other hand $G$ has exactly two components, for otherwise 
$\gamma(G)\geq3$ or $G=K_{1}$, contradicting the fact that $\rho_{o}(G)>1$. Thus $G=\bar P‎_{3}‎$.\\
The results now follow by considering Case $1$ and Case $2$.
\end{proof}


\end{document}